\documentclass[12pt]{amsart}
\usepackage{hyperref}
    
    \usepackage[utf8]{inputenc}
\usepackage[english]{babel}

\usepackage[margin=0.8in,footskip=0.5in]{geometry}

\usepackage{amsthm,amssymb,amsmath,enumerate,graphicx, tikz}
\usepackage{amscd}
\usepackage{setspace}
\usepackage{comment}
\newtheorem{theorem}{Theorem}[section]
\newtheorem{challenge}{Challenge}[section]
\newtheorem*{theorem*}{Theorem}

\newtheorem{lemma}[theorem]{Lemma}
\newtheorem{observation}[theorem]{Observation}

\newtheorem{claim}[theorem]{Claim}
\newtheorem{conjecture}[theorem]{Conjecture}
\theoremstyle{definition}
\newtheorem{definition}[theorem]{Definition}

\theoremstyle{remark}
\newtheorem{remark}[theorem]{Remark}
\newtheorem{example}[theorem]{Example}

\newcommand{\cf}{\mathcal{F}}
\newcommand{\cg}{\mathcal{G}}
\newcommand{\cs}{\mathcal{S}}

\newcommand{\blambda}{\bar{\lambda}}

\newcommand{\vv}{\vec{v}}

\newcommand{\vspan}{\operatorname{span}}

\newcommand{\floor}[1]{\left\lfloor #1 \right\rfloor}
\newcommand{\ceil}[1]{\left\lceil #1 \right\rceil}

\newcommand{\bF}{\mathbb{F}}
\newcommand{\C}{\mathcal{C}}
\newcommand{\cc}{\mathcal{C}}

\newcommand{\M}{\mathcal{M}}
\newcommand{\A}{\mathcal{A}}
\newcommand{\N}{\mathcal{N}}
\newcommand{\K}{\mathcal{K}}

\newcommand{\PP}{\mathcal{P}}

\newcommand{\cb}{\mathcal{B}}

\newcommand{\V}{\mathcal{V}}

\newcommand{\R}{\mathbb{R}}
\newcommand{\B}{\mathcal{B}}

\newcommand{\F}{\mathcal{F}}

\newcommand{\G}{\mathcal{G}}

\renewcommand{\c}{\mathbf{c}}
\renewcommand{\V}{\mathcal{V}}

\newcommand{\vx}{\vec{x}}
\newcommand{\vu}{\vec{u}}

\newcommand{\rank}{\textup{rk}}

\newcommand{\conv}{\textup{conv}}
\newcommand{\cone}{\textup{cone}}
\newcommand{\cm}{\mathcal{M}}
\newcommand{\cn}{\mathcal{N}}
\newcommand{\cp}{\mathcal{P}}

\author[Aharoni]{Ron Aharoni} \thanks{R. Aharoni:
Department of Mathematics, Technion, Israel
\url{ra@tx.technion.ac.il}.
Ron Aharoni is Supported by the Israel Science Foundation (ISF) grant no.\ 2023464 and the Discount Bank Chair at the Technion. This paper is part of a project that has received funding from the European Union's Horizon 2020 research and innovation programme under the Marie Sklodowska-Curie grant agreement no.\ 823748.
}
\author[Briggs]{Joseph Briggs} \thanks{J. Briggs (corresponding author, correspondence to: joseph.guy.briggs@gmail.com). Department of Mathematics, Auburn University, AL, USA \url{jgb0059@auburn.edu}}

\begin{document}

\thispagestyle{empty}

\title{Choice functions}
\maketitle
 
\begin{abstract}
This is  a survey paper on rainbow sets (another name for ``choice functions''). The main theme is  the distinction between two types of choice functions: those having a large (in the sense of belonging to some specified filter, namely closed up set of sets) image, and those that have a large domain and small  image, where ``smallness'' means belonging to some specified complex (a closed-down set). The paper contains some new results: (1)  theorems on scrambled versions, in which  the sets are re-shuffled before choosing the rainbow set, and (2) results on weighted and cooperative versions -  to be defined below.

\end{abstract}

\section{Introduction}

\begin{flushleft} 
{\em To make a prairie it takes a clover and a bee \\
One clover, and a bee.\\
And revery.\\
The revery alone will do\\
if bees are few.\\
\ \\
Emily Dickinson, Poem 1755}

\end{flushleft}
\bigskip

To make a mathematical field it takes  a fundamental theorem, and a good conjecture. The field that is the topic of this article satisfies both conditions. The theorem is Hall's marriage theorem, and the conjecture is Ryser's conjecture on transversals in Latin squares.

\begin{definition}
Let $\cs$ be a collection  of subsets of a set $V$.  An $\cs$-partial choice function is a  function $f: \cs' \to \bigcup \cs$ for some 
 $\cs' \subseteq \cs$ , such that $f(S) \in S$ for all $S \in \cs'$. Its image  is called an $\cs$-{\em partial rainbow set}, and if $\cs'=\cs$ the image is called  a {\em full} rainbow set.   
\end{definition}

\begin{remark}\hfill

 (a) When speaking of a rainbow set,  we implicitly retain the function itself, namely we remember where each representative came from.
 
 (b) The representation is  assumed to be injective, namely each representation is by a distinct element.
\end{remark}

The sets in $\cs$ are thought to ``color'' the elements, hence the name ``rainbow''. 
Hall's marriage theorem (anticipated by theorems of Frobenius and K\"onig) is:

\begin{theorem}
 Let $\cs$ be a family of sets. If $|\bigcup \cs'|\ge |\cs'|$ for  every $\cs' \subseteq \cs$, then there exists a full injective choice function. 
\end{theorem}

The conjecture to which we allured is usually stated in the terminology of Latin squares. An $n \times n$ matrix $L$  is called a {\em  Latin square} if its entries belong to $[n]:= \{1, \ldots ,n\}$, 
and  every row and every column consists of  distinct elements. A (partial) {\em transversal} in $L$ is a set of distinct entries, lying in different rows and different columns. It is called {\em full} if its size is $n$.

\begin{conjecture}\label{brs}\cite{brualdi, stein}
Every $n \times n$ Latin square has a transversal of size $n-1$. If $n$ is odd, then there exists a full transversal. 
\end{conjecture}

The best result to date is by Keevash-Pokrovsky-Sudakov-Yepremyan:
\begin{theorem}\cite{KPSY}
 Every $n \times n$ Latin square has a transversal of size $n-O(\log n/ \log \log n)$.
\end{theorem}

The paper is not presumed to be comprehensive in any way, but rather to touch on some main themes. It also contains a few new results: on ``scrambled'' versions (to be defined in Section \ref{sec:scrambling}), on cooperative versions (Section \ref{coopsection}) and on weighted versions (Section \ref{sec:weighted}).

\section{A general theme - big in one sense, small in another}

 Hall's marriage theorem concerns 
a well-known theme: the existence of an object satisfying two opposing requirements. It should be large in one sense, and small in another. For example, in the marriage theorem the choice function should be large on the side of the sets - represent all of them, and small on the elements side - being injective. 

A hypergraph $H$ is called a {\em simplicial complex}, or just a {\em complex} if it is  closed down, namely $e \in H, ~f \subseteq e$ imply $f \in H$. A simplicial complex is a {\em matroid} if all its maximal edges are also of  maximal size, and this is true also for every induced sub-complex.  An edge of the matroid is said to be {\em independent}, a set containing a maximal edge  is said to be {\em spanning}, and a set that is both independent and spanning is a {\em base}.  Given a complex $\C$ we think of its elements as ``small'', and a spanning set in a matroid is considered ``large''. Of particular interest from a combinatorial point of view are {\em partition matroids}, defined by a partition of the ground set, where a set belongs to the matroid if it meets every part in at most one vertex. 
A basis of a partition matroid is just a full rainbow set of the parts.

In a  more general setting  a matroid $\M$ is given on $V:=\bigcup \cf$, and to the   injectivity requirement  a requirement is added   that the rainbow set $R$ satisfies $R \in \M$. In Hall's theorem $\M$ is the partition matroid whose parts are the $v$-stars, $v \in V$. Hall's theorem then generalizes to the following \cite{rado}: 

\begin{theorem}[Rado]
\label{rado}
There exists a full rainbow set $R$ belonging to $\M$
if and only if 
 $rank_\M(\cf_I) \ge |I|$ for every $I \subseteq [m]$.\end{theorem}

Edmonds \cite{edmonds} offered a further generalization. 
Consider the bipartite graph whose one side is $\cs$ and the other is $\bigcup \cs$, where $S \in \cs$ is connected to $x \in \bigcup \cs$ if $x \in S$.  The two matroids are the two  partition matroids on  the edge set $E$ of the graph, whose parts are the stars
in  the two respective sides. 
In this setting, a set of edges contains a full choice function if it is spanning in the first matroid, and it is injective if it belongs to the second matroid. 

In a further generalization,  one of the matroids is replaced by a general simplicial complex.

\begin{definition}
    Let $\cm, ~\cc$ be a matroid and a complex, respectively, on the same ground set. An  $(\cm, \cc)$-{\em marriage} is a set spanning in $\cm$ and belonging to $\cc$. An  {\em  $(\cm, \cc)$-rainbow set} is a set  belonging $\cm \setminus \C$.
\end{definition}

The ``if'' direction of Theorem \ref{rado} still holds in this setting, with ``rank'' replaced by ``coonectivity'', to be defined below. 

We shall be particularly interested in  rainbow {\em matchings}. 
Given a hypergraph on a set $V$ and a partition of its edge set, a {\em partial rainbow matching } is a partial choice function on the parts, whose image is a matching (=set of disjoint edges). A more general notion is that of  rainbow independent sets in  a graph. 
Matchings are the special case, of independent sets in a line graph.

Given a class $\C$ of hypergraphs, we write $(a,b) \to^\C c$ if every family of $a$ matchings of size $b$ in a hypergraph belonging to  $\C$ has a rainbow matching of size $c$. Let $\cg$ be the class of all graphs, and $\B$  the class of bipartite graphs.

Conjecture \ref{brs} has a  generalization, by Berger and the first author, that can be formulated in this terminology. We shall refer to it as the ``A-B conjecture'':

\begin{conjecture}
\label{abab}\cite{ab1}
$(n,n) \to ^\cg n-1$. 
\end{conjecture}

In words: $n$ matchings of size $n$ in any graph 
have a rainbow matching of size $n-1$.

In \cite{ab1} this was formulated only for bipartite graphs, but we know no counterexample also for general graphs.

In the bipartite case, doubling the number of matchings in the condition yields a well-known theorem of Drisko:

\begin{theorem} \label{drisko}\cite{drisko, ab1} 
$(2n-1,n) \to ^\B n$.
\end{theorem} 
In words: 
$2n-1$ matchings of size $n$ in a bipartite graph have a rainbow matching of size $n$. 
A very nice recent result, using algebra for its proof, is:

\begin{theorem} \cite{cst} \label{thm:cst}
Let $\C_r$ be the class of $r$-uniform hypergraphs. 
$\to^{\C_r} t$. 
\end{theorem}

Other conjectures in the spirit of Conjecture \ref{abab} are: 
\begin{conjecture}
\label{ababa}\hfill
\begin{enumerate}
   \item
$(2n,n)\to^\cg n$. If $n$ is odd then $(2n-1,n)\to^\cg n$.
\item
$(n,n+1)\to^\B n$.
\item 
$(n,n+2) \to^\cg n$.

\end{enumerate}
\end{conjecture}

An example showing that $(n.n+1)\not \rightarrow^\cg n$ is 
given by three matchings $M_i,~i\le 3$ on a set $\{a,b,c,d,a',b',c',d'\}$, where $M_1=\{ab,cd,a'b',c'd'\}$, $M_2= \{ac,bd,a'c',b'd'\}$ and  $M_3= \{ad,bc,a'd',b'c'\}$. However, we do not know of any such example for larger $n$.

 Pokrovskiy \cite{pok} proved that  $(n,n+o(n)) \to ^\B n$ and Rao, Ramadurai, Wanless and Wormald \cite{grww} gave better asymptotics in the case of simple graphs. Keevash and Yepremyan  \cite{ky} proved that $n$ matchings of size $n+o(n)$, repeating each edge at most $o(n)$ times, 
have a rainbow matching of size $n-\text{const}$.
It is not hard to show that $(n,n) \to^\cb n-o(n)$, 
while the same statement for general graphs is still open, see  Section \ref{sectiongeneralgraphs} below.

Interestingly, part (2) of  Conjecture \ref{ababa} implies Theorem \ref{drisko} (of course, the other direction would be more desirable). Given $2n-1$ matchings of size $n$ in a bipartite graph, double each of them by joining it to its copy on a disjoint vertex set $V'$. We now have at hand $2n-1$ matchings of size $2n$, so assuming  $(2n-1,2n)\to ^\B 2n-1$, we get a rainbow matching of size $2n-1$, of which by the pigeonhole principle $n$ edges must be  in the same copy  - $V$ or $V'$. Similarly, Conjecture \ref{abab} implies part (1) of the conjecture.

For bipartite graphs, (2) implies (1). If true, (1)  is a strengthening of Conjecture \ref{brs}. To see this, note that for every $i \in [n]$ the set of entries that are equal to $i$ forms a matching between the rows and the columns.

For  a non-decreasing sequence $\sigma=(a_1, \ldots ,a_k)$  of natural numbers write $\sigma \to n$ if for every set of  families $F_i$ of matchings in a bipartite graph, where $|F_i|=a_i$, there exists a rainbow matching of size $n$. We say then that $\sigma$ is $n$-{\em coercive}. So, Conjecture \ref{abab} is that $(n+1, \ldots ,n+1) \to n$ ($(n+1)$-fold repetition). We do not even know a counterexample to $(n,\ldots n,, n+1, \dots, n+1)\to
n$ ($n,n+1$ respectively repeated $\ceil{n/2},\floor{n/2}$ times). 

In \cite{abkz} the following  strengthening of  Theorem \ref{drisko} was proved:

\begin{theorem}\label{stairs}
 $(1, 2, \ldots n,n,\ldots ,n)\to n$ ($n$ repeated $n$ times).
\end{theorem}

The only known proof is topological. This sequence is {\em minimal $n$-coercing}, in the following sense:  reducing any term makes the sequence non-$n$-coercing. Another example of a minimal $n$-coercing sequence is $(1,3,\dots,2n-3,2n-1)$. 

\begin{conjecture}\label{bold}\hfill
\begin{enumerate}
    \item If $\sigma$ is minimal $n$-coercing then its length is at most $2n-1$. 
    
    \item An $n$-coercing sequence contains an $n$- coercing subsequence  of length at most $2n-1$.

     \item If $\sigma$ is not $n$-coercing then there is a set of matchings witnessing this, all contained in a disjoint union of  cycles. 
     
     \item 

If $(a_1. \ldots ,a_k) \to n$ then $\sum_{i=1}^k a_i \ge n^2$.
\end{enumerate}
\end{conjecture}

The sequence $(2,4,4)$ is not 3-coercing, and the only example witnessing this is a disjoint union of two cycles of length $4$, so part (3) of the conjecture cannot be strenthened to ``one cycle''.  Conjecture \ref{abab}  is easily seen to be true in cycles.  

\section{Two dual,  topologically-formulated,   theorems }

The {\em rank} of a complex is the largest size of an edge. 
A complex of rank $k$ has a unique (up to homeomorphism) geometric realization in $\mathbb{R}^{2k-1}$ - think of the realization of a graph ($k=2$) in $\mathbb{R}^3$.
A  parameter connecting topology and combinatorics is {\em (homotopic) connectivity}, which is the minimal dimension of a hole in the geometric realization. Its homological version is the minimal index of a non-zero homology group, plus 1. The two are not identical - homological connectivity is at least as large as homotopic connectivity. The homotopic connectivity is denoted by $\eta(\C)$ ($\eta=\infty$ if there are no holes). 
 Rainbow problems and marriage problems  give rise to two mirror requirements on the holes. In the first, it is beneficial (namely contributing to the existence of rainbow sets) not to have small dimensional holes. In the second, it is useful not to have large dimensional holes. We denote by $\lambda(\C)$ the largest dimension of a hole ($0$, if there is no hole), and by $\blambda(\C)$ the maximum of $\lambda(\C[S])$ over all subsets $S$ of $V(\C)$. 
 It is called the ``Lerayness'' of $\C$.
\begin{example}
Let $S^n$ be (a triangulation of) the $n$-dimensional sphere, and $B^n$ (a triangulation of) the $n$-dimensional ball. Then  $\eta(S^n)=\lambda(S^n)=n+1, ~\eta(B^n)=\infty$.
\end{example}

\begin{remark}
 It can be shown that $\eta$ is the minimal dimension of an empty {\em sphere}. By contrast, $\lambda$ is not necessarily the largest dimension of an empty sphere - the ``holes'' can be more general. For example, for a torus $\lambda=3$, while $\eta=2$, since there is an empty $S^1$. 
\end{remark}

The two parameters are connected via an operation called ``Alexander duality''. The Alexander dual $D(\C)$ of a complex $\C$ is $\{A^c \mid A \not \in \C\}$. 

\begin{theorem}[e.g. \cite{bt}]\label{alexander}
$\blambda(\C) =|V(\C)|-\eta(D(C))-1$. 
\end{theorem}

The connection between these notions and choice functions is given by two theorems. One is ``Topological Hall" (implicit in \cite{ah}, first stated explicitly, quoting the first author,  in \cite{meshulam2}), whose matroidal form is:

\begin{theorem}\label{matcomp}\cite{matcomp}
 If $\eta(\C[S]) \ge rank(\cm.S)$ for every $S \subseteq V$ then there exists a $(\cm, \cs)$-marriage.
\end{theorem}

Here $\cm.S$ is the contraction of $\cm$ to $S$. ``Topological hall'' is the case in which $\cm$ is a partition matroid. (The disjointness of the parts in the partition matroid is attained by duplicating vertices, if necessary.)

 The fundamental topological theorem regarding the second type of choice functions is that of Kalai-Meshulam, a special case of which is: 

\begin{theorem}[Kalai--Meshulam \cite{km}]\label{kmleray}\label{km}
    If $\blambda(\C) \le d$  then every $d+1$ sets in $\C^c:=2^V \setminus \C$ have a rainbow set belonging to $\C^c$.
\end{theorem}

As in Theorem \ref{drisko} and Conjecture \ref{abab}, the format is that of ``many sets not in $\C$ have a rainbow set not in $\C$''. Theorems \ref{matcomp} and \ref{km}
are derivable from one another, using Theorem \ref{alexander} and  noting that $A \in \cm \setminus \C$ if and only if $A^c$ is spanning in $\cm^*$ and belongs to $D(C)$. ($\cm^*$, the dual of $\cm$, consists of those sets whose complements span $\cm$).

 A classical theorem of this type is the B\'ar\'any-Lov\'asz colorful version of Caratheodory's theorem. For a set $V=\{\vec{v}_i, ~i \in I\}$ of vectors in $\mathbb{R}^d$ let  
$$cone(V)= \{\sum_{i \in I} \alpha_i \vec{v}_i, ~\alpha_i \ge 0 ~\text{for all } i \in I \}$$
and
$$conv(V)= \{\sum_{i \in I} \alpha_i \vec{v}_i, ~\alpha_i \ge 0 ~\text{for all } i \in I, ~\sum_i\alpha_i=1 \}$$

\begin{theorem}\label{barany}

[B\'ar\'any, \cite{barany}]
\begin{enumerate} 
\item
If $S_1, \ldots ,S_d$ are sets of vectors in $\mathbb{R}^d$ satisfying $\vec{v} \in cone(S_i)$ for all $i \le d$ then there exists a rainbow set $S$
such that $\vec{v} \in cone(S)$.

\item
If $S_1, \ldots ,S_{d+1}$ are sets of vectors in $\mathbb{R}^d$ satisfying $\vec{v} \in conv(S_i)$ for all $i \le d+1$ then there exists a rainbow set $S$
such that $\vec{v} \in conv(S)$.

\end{enumerate}
\end{theorem} 

The complex $\cc$ consists here of sets not containing $\vec{v}$ in their convex hull (or cone). 
The case of Theorem \ref{drisko} in which the matchings reside in $K_{n,n}$ can be  derived from Theorem \ref{barany} - we omit the proof.

\section{General graphs}\label{sectiongeneralgraphs}
Recall Conjecture \ref{ababa} (1):  
$(2n,n)\to^\cg n$. 
   The fractional version of this conjecture was proved in \cite{ahj}. It says that 
   any $2n$ matchings of size $n$, in a general graph, have a rainbow set of edges possessing a fractional matching of size $n$.    The best result  so far is $(3n-3,n)\to n$ \cite{abkk}. Holmsen and Lee \cite{hl} gave a topological proof to the weaker result $(3n-2,n)\to n$, that is stronger in another sense - it yields a cooperative version: 
   
   \begin{theorem}
    Any collection of $3n-2$ non-empty sets of edges in any graph, satisfying the condition that every pair of them contains in its union a matching of size $n$, has a rainbow matching of size $n$, 
   \end{theorem}

   We shall return to the theme of cooperation in Section \ref{coopsection}.
   
   Recently, Chakraborti and Loh \cite{cl} proved an asymptotic version of the conjecture,
   when the matchings are disjoint. 

Recall also Conjecture \ref{abab}:
$(n,n)\to^\cg n-1$. 
 Woolbright  \cite{woolbright} proved $(n,n) \to^\B n-\sqrt{n}$. But in general graphs the situation is gloomier: 
 we do not even know how to show $(n,n) \to^\G n-o(n)$. 
 The best result so far is in \cite{abcz}, the existence of a rainbow matching of size at least $\frac{2n}{3}-1$. The result, with $\sqrt{2n}$ replacing $o(n)$, would follow from the following conjecture:

\begin{conjecture}
Let $F$ be a matching in a graph, and let $\cf_i$ be families of disjoint augmenting $F$-alternating paths. If $\sum |\cf_i| \ge 2|F|$ then there exists a set $K$ of edges, each taken from a path in a different $\cf_i$, that contains an augmenting $F$-alternating path. \end{conjecture}

A proof of this fact when the $\cf_i$s are each a singleton path is at the core of the proof in \cite{abchs}. The non-alternating version of this conjecture was proved in  \cite{abcz}:

\begin{theorem}
Let $F$ be a set of vertices in a graph, and let $\cf_i$ be families of disjoint paths, starting and ending outside $F$.  If $\sum |\cf_i| \ge 2|F|$ then there exists a set $K$ of edges, each taken from a path in a different $\cf_i$, that contains a path starting and ending outside $\bigcup F$. \end{theorem}

The conjecture would follow from $(n,n+o(n)) \to^\G n$, but presently this is  known only for sets of edge-disjoint matchings.

\section{Covering the vertex set by rainbow sets}\label{covering}

The next step goes one level higher: asking for many disjoint full rainbow sets.  
The best known conjecture of this sort  is Rota's conjecture, that says that given $n$ independent sets of size $n$ in a matroid $\cm$, there are $n$ disjoint $\cm$-independent rainbow sets. A recent startling development on the conjecture was its asymptotic proof by Pokrovskiy \cite{pokrota} - there are $n$ partial rainbow independent sets covering $n^2-o(n)$ of the elements.

 For a hypergraph $H$, denote by $\rho(H)$ the minimal number of edges covering $V(H)$. A conceivable strengthening of Rota's conjecture  is that the sets need not be in $\cm$ - it may suffice that they are in $\M$ only after scrambling. Namely - does it suffice to assume that $\rho(\cm) \le n$? 
 In \cite{ak} counterexamples were given for $n$ odd, but the following may  be true: 

\begin{conjecture}\label{scrambledrota}\hfill
\begin{enumerate}
    \item

If $\rho(\cm)=n$ and $V(\cm)$ is partitioned into $n$ sets of size $n$, then  $V(\cm)$ can be partitioned into $n+1$ rainbow $\cm$-sets. 

\item If $n$ is even then $n$ rainbow $\cm$-sets suffice.
\end{enumerate}
\end{conjecture}

Admittedly, the $n$-even part of the conjecture is rather daring, and is safer to be conjectured just for linear matroids. But it matches a  difference in the original conjecture: in the linear case, Rota's conjecture is known  for infinitely many even $n$'s, while for $n$ odd it is known only for $n=3$ (a result of R. Huang).

Part (1) of the conjecture   is a special case of:

\begin{conjecture}\label{matroid_s-g}
If $\M, \N$ are two matroids on the same ground set, then $\rho(\M \cap \N) \le \max(\rho(\M), \rho(\N))+1$.
\end{conjecture} 
Conjecture \ref{scrambledrota}(1) 
is the  case in which one of the two matroids is a partition matroid. 
The conjecture is known even without the ``$+1$'', when $\cm$ and $\cn$ are truncated partition matroids, namely the intersection of a partition matroid with a uniform matroid (the latter being the collection of sets of size $\le k$ for some $k$) - this is a re-formulation of  K\"onig's edge-coloring theorem.  In particular, this is true when $rank(\cm), rank(\cn) \le 2$ - it is easy to show that a rank $2$ matroid is of this type.

It is not hard to show that if $\M,\N$ are matroids, then $\max(\rho(\M), \rho(\N)) = \rho^*(\M \cap \N)$, and therefore Conjecture \ref{matroid_s-g} can be re-formulated as:
\begin{conjecture}
$\rho(\M \cap \N) \le \rho^*(\M \cap \N)+1$
\end{conjecture}

This is a close relative of the famous Goldberg-Seymour conjecture, stating the same inequality where the intersection of two matroids is replaced by the matching complex of a multigraph. 
The Goldberg-Seymour conjecture has recently been claimed to be solved \cite{gs}, and the methods used there (mainly relying on alternating paths) may be relevant also to this conjecture. 

In fact, there is a common generalization to this conjecture and the Goldberg-Seymour conjecture. Its objects are $2$-{\em polymatroids}. A hypergraph $\PP$ on a vertex set $V$ is called a $2$-polymatroid if there exists a matroid $\M$ is on $V \times \{1,2\}$, such that $\PP=\{A \mid A\times \{1,2\}\in \M\}$. In \cite{2poly}
the following was proposed:

\begin{conjecture}
If every circuit (minimal non-$\M$ set) meets every pair $\{(v,1),(v,2)\}$ in at most one element, then 
$\rho(\PP) \le \rho^*(\PP) +1$.
\end{conjecture}

In \cite{2poly} this is proved when $\rho^*(\PP)\le 2$.

In \cite{matcomp} the following was proved:

 \begin{theorem}\label{2factor}
 
 $\rho(\M \cap \N) \le 2 \max(\rho(M), \rho(N))$. 
 \end{theorem}
 
Can an analogous packing result (rather than covering) be proved? It is 
easy to prove that there exists a set in $\M \cap \N$
of size $\lceil \frac{|V|}{\max(\rho(M), \rho(N))}\rceil$. Can we prove the existence of ``many'' disjoint such sets?

 A beautiful  
strengthening of Theorem \ref{2factor} was suggested in [\cite{berczi}, Conjecture 1.10]:

\begin{conjecture}\label{partitionapproximation} 
Every matroid $\cm$ contains a partition matroid $\cp$ with $\rho(\cp)\le 2\rho(\cm)$.
\end{conjecture}
 
 Theorem \ref{2factor} will follow from this conjecture by K\"onig's edge-coloring theorem.

Scrambling takes us to a happy hunting ground, with an abundance of problems. Can we prove the existence of two disjoint $(\cm,\cp)$-marriages in the scrambled setting of Rota? In the original setting of Rota's conjecture  
it is known \cite{halfrota} that there exist $\frac{n}{2} $ disjoint such marriages.


Matchings are independent sets in line graphs, and our problems can be generalized, to independent rainbow sets in general graphs. A  well-known conjecture on such sets (see, e.g.,  \cite{hax}) is:

\begin{conjecture}\label{strongcoloring}
Let $G$ a graph with maximal degree $d$, and let $\cs$ be a partition of $V(G)$ into sets of size $2d$. Then $V(G)$ can be covered by $2d$ independent rainbow sets. 
\end{conjecture}

 In \cite{rainbowindependent} the following generic question was studied: given a class of graphs, how many independent sets of size $n$ do you need to procure a rainbow independent set of size $m$?
  The following conjecture was posed there, with some partial results:
  
  \begin{conjecture}\label{degreek}
  $\left\lceil \frac{k+1}{n-m+2} \right\rceil(m-1) +1$
  independent sets of size $n$ in a graph with maximal degree $k$ have a rainbow independent set of size $m$. 
  \end{conjecture}

Conjecture \ref{degreek} is open  even  for $k=2$, in which case independent sets can be viewed as matchings - taking us back to rainbow matchings. The special case $k=2$,
 $m=n$
 says that $n$ independent sets of size $n$ in the disjoint union of cycles and paths have a rainbow independent set of size $n-1$. Switching from $m=n-1$ to $m=n$ causes the expression $\lceil \frac{k+1}{n-m+2}\rceil
 $ to jump from 
$\lceil \frac{3}{3} \rceil=1$ to $\lceil \frac{3}{2}\rceil =2$.
This is one point of view from which to explain the jump from $n$ to $2n-1$ matchings of size $n$,
when the desired size of the rainbow matching goes from $n-1$ to $n$.

 In \cite{rainbowindependent} examples were given to show that,  if true,  the conjecture is sharp. Also ``half'' of the conjecture is proved -  
  if  $m \le n$ then $k(m-1)+1$ independent sets of size $n$ in a graph of maximal degree $k$ have a rainbow independent set of size $m$.

\section{Rainbow  cycles and rainbow  spanning sets}

We next turn to problems of algebraic nature. 
 The following fact  can be proved using Theorem \ref{km}, but it also has an easy direct proof:

\begin{observation}
If $rank(\C) \le n$ then any collection  of $n+1$ sets not belonging to $\C$ has a rainbow set not belonging to $\C$.
\end{observation}

A bit less trivial: 
\begin{theorem}\label{main}\label{odd}\label{rainbowodd}
   Any family $\A=(A_1, \ldots ,A_n)$ of edge sets of odd cycles on a set of size $n$ has a rainbow odd cycle. \end{theorem}
    
     This is a special case of a more general result:
    
    \begin{theorem}\label{spanning}
Let $\M$ be a matroid of rank $n$ and let $v\in V(\M)$. Any collection  $A_1, \ldots ,A_n$ of sets spanning $v$ has a rainbow set spanning $v$.
\end{theorem}

 Theorem \ref{odd} follows from Theorem \ref{spanning} and the following lemma: 

\begin{lemma}\label{l.bipartitevectorspanning}
$G$ is bipartite if and only if  $$(\vec{0},1) \not \in \text{span}(\{(\chi_e,1) \mid e \in E(G)\})$$ 
\end{lemma}
Here the vectors are taken in $\mathbb{F}_2^n$. Note that the condition is just that of affine spanning.

To prove Theorem \ref{odd}, take the matroid $\M$ in Theorem \ref{spanning} to be the linear matroid over $\mathbb{F}_2$ and replace every edge $e$ by the vector $(\chi_e,1)$.

For $n$ odd, a Hamilton cycle on $n$ vertices, taken $n-1$ times, shows that Theorem \ref{odd} is sharp, namely  $n-1$ odd cycles do not suffice. It was shown in \cite{abhj} that the only examples showing sharpness are cacti-like graphs, generalizing the above example.

One of the best-known 
conjectures in graph theory  is the Caccetta - H\"aggkvist conjecture:

\begin{conjecture}\cite{ch}\label{kh}
In a digraph on $n$ vertices and minimum out-degree at least $\frac{n}{r}$ there is a directed cycle of length $r$ or less.
\end{conjecture}  

In \cite{adh} this was given a rainbow generalization.

\begin{conjecture}
In an undirected graph on $n$ vertices, any collection of $n$ disjoint sets of edges, each of size at least $\frac{n}{r}$, has a rainbow cycle of length $r$ or less.
\end{conjecture} 

In \cite{sparse} the conjecture is proved for $r=2$, while the original C-H conjecture is known up to $r=5$ - so there is a lot of room for progress.

    \section{ Rainbow paths in networks and their relation to rainbow matchings}
    
   A {\em network} is a triple $(D,S,T)$, where $D$ is a digraph on a vertex set $V$, and $S,T$ are disjoint sets, of  ``sources''  and ``targets'', respectively  It is assumed that no edge goes into $s$, and no edge goes out of $t$.

   In \cite{akz}
   the following result was proved, as a step towards proving Theorem \ref{drisko}: 
    
    \begin{theorem}\label{networks}
    If $|V(D)\setminus (S \cup T)|=n$ then any family $\cf$ of $n+1$   $S-T$-paths has a rainbow $S-T$ path. 
    \end{theorem}
  The two theorems have a common generalization. 
   Given a set $\F$ of edges let $\nu^P(\cf)$ be the maximal size of a family of $S-T$ disjoint paths with edges from $\F$.

\begin{theorem}\label{manypaths}
    Let $\N$ be a network with $|V^\circ(\N)|=q$,  and let  $p$ be an integer.
Let $\F=(F_1, \ldots,F_{2p-1+q})$ be a family  of sets of edges, satisfying $\nu^P(F_i) \ge p$ for all $i \le 2p-1+q$. Then there exists an $\F$-rainbow set $R$ with $\nu^P(R) \ge p$. 
    \end{theorem}
    
The case $q=0$ is Theorem \ref{drisko}, and the case $p=1$ is Theorem \ref{networks}.

The proof requires some definitions.

Given a network $\N$, define a bipartite graph $B(\N)$ whose vertex set is $\{v' \mid v \in S \cup V^\circ\} \cup \{v'' \mid v \in T \cup V^\circ\}$, and whose edge set is $\{u'v'' \mid uv \in E(\N)\} \cup W$, where $$W=\{x'x'' \mid x \in V^\circ\}.$$ 

 Here $v'$ is a ``sending'' copy of $v$, and $v''$ is an ``absorbing'' copy of $v$.

For a matching $M$ in $B(\N)$ let $\psi(M)=M \setminus  W$. 

A {\em linearish arborescence} is a digraph in which every vertex has in-degree and out-degree at most one, meaning that it is a collection of directed paths and cycles.
Given a linearish arborescence $L$ in  $\N$, 
let $\phi({L})$ be the matching  $\{x'y'' \mid xy \in E(L)\}  \cup \{x'x'' \mid x \not \in V(L)\}$ in $B(\N)$.

 For a linearish arborescence $L$ let $L_{CIRC}$ be the  the set of cycles in $L$, $L_S$ the set of  paths in $L$ starting at $S$,    $L_T$  the set of paths in $L$ ending  at $T$, $L_{ST}= L_S \cap L_T$. By definition, $|\nu^P(L)|=|L_{ST}|$.  
 
 \begin{claim}\label{nup}
    $|L_{ST}| = |\phi(L)| -|V^\circ(\N)|+|L\setminus (L_S \cup L_T)|$.

    \end{claim}
   To see this, note: 
    \begin{enumerate}
        \item If $C \in L_{CIRC}$ then $|V(C)|=|E(C)|$, 
        \item if $P\in L_{ST}$ then $|E(P)|=|V^\circ(P)|+1$, so $P$  contributes to $\phi(L)$ $1$ more than to $V^\circ$, 
        \item paths in $(L_S \cup L_T) \setminus L_{ST}$ contribute the same number to $\phi(L)$ and to $V^\circ$, and 
        
        \item paths in $L\setminus (L_S \cup L_T)$ contribute $1$ less to  $\phi(L)$  than to $V^\circ$. 
    \end{enumerate}



\begin{proof} [Proof of Theorem \ref{manypaths}]
    For every $i \le 2p-1+q$ let 
$M_i=\phi(F_i)$. By Claim \ref{nup} (taking $L=F_i)$  $|M_i|=p+q$.    Consider the sequence of $2(p+q)-1$ matchings 
$W_1, W_2, \ldots W_{q},M_1, \ldots ,M_{2p-1+q}$, where $W_j=W$ for all $j \le q$. By Theorem \ref{stairs} (letting $n=p+q$), these matchings  have a rainbow matching $R$ of size $p+q$. Let $L=\psi(R)$. By Claim \ref{nup} 
$\nu^P(R) \geq |F_i| - q=p$, and clearly the paths in $R$ witnessing this are constructed from edges in $M_i$. 
\end{proof}

\section{Scrambling}\label{sec:scrambling}

We next return to a notion from Section \ref{covering}: 
\begin{definition}
    Given two families (that is, multisets) $\F,\cs$ of subsets of the same ground set, we say that $\cs$ is a \emph{scrambling} of $\F$ if $\bigcup \F=\bigcup \cs$ (as multisets), and that it is an {\em $n$-scrambling} if every set in $\cs$ has size at most $n$. 
\end{definition}

A straightforward application of Rado's theorem (Theorem \ref{rado}) yields:
\begin{theorem}
 An $n$-scrambling of $n$ independent sets of size $n$ in a matroid have an independent full rainbow set. 
\end{theorem}

 In \cite{abbsz} the following was proved:
 
 \begin{theorem}\label{scrambledmatchings}
  If $\F$ is a family of $n^2-n/2$ matchings of size $n$ in a bipartite graph and $\cs$ is an $n$-scrambling of $\F$, then $\cs$ has a rainbow matching of size $n$.
 \end{theorem}



The proof uses Topological Hall. This may possibly be improved:

\begin{conjecture}
Let $n \geq 4$. If $\F$ is a family of $\binom{n}{2}+1$ matchings of size $n$ in a bipartite graph and $\cs$ is an $n$-scrambling of $\F$, then $\cs$ has a rainbow matching of size $n$. 
\end{conjecture}
 
If true, this is best possible. In \cite{abbsz}  the following was shown: 

\begin{theorem}
 For $n \geq 4$, there exists an $n$-scrambling of  $\binom{n}{2}$ matchings of size $n$ in a bipartite graph, with no rainbow matching of size $n$.
\end{theorem}

We can prove the following scrambled version of Theorem \ref{networks}:

\begin{theorem}\label{scramblednetworks}
Suppose $|V(D) \setminus \{s,t\}|=k$, and $\F$ is a family of $>\frac{nk}{2}$ $s-t$ paths. If $\F'$ is an $n$-scrambling of $\F$ then $\F'$ has a rainbow $s-t$ path.
\end{theorem}

The main tool in the proof will be the following class of objects.

\begin{definition}
    An $n$-\emph{fold source tower} in $D$ is a sub-digraph $S$ containing $s$, equipped with an ordering of its vertices $s=v_0,v_1,\dots $ such that the set of edges to $v_i$ from all previous $v_j$, denoted $S^-(v_i)$, is of size at least $n$ for every $i>0$. In particular, there are $\geq n$ edges from $s$ to $v_1$. Similarly, $T\subset D$ containing $t$ is an $n$-\emph{fold target tower} if $V(T)$ has an ordering where every $u \in T$ has a set $T^+(u)$ of at least $n$ edges to all previous vertices. 
\end{definition}

\begin{proof} (Theorem \ref{scramblednetworks}):
Let $S,T$ be a maximal pair of vertex-disjoint $n$-fold source and target towers.

Let $W=V(D)\setminus (S\cup T)$.
For each $w\in W$, denote by $e(S,w), e(w,T)$ the sets of edges in $D$ from $S$ to $w$ and from $w$ to $T$ respectively. 

Suppose first that there is some $w\in W$ with $|e(S,w)|+|e(w,T)| > n$. Note that both $e(S,w)$ and $e(w,T)$ are nonempty, otherwise the other would have size at least $n+1$ and $w$ could be added to the corresponding $n$-fold tree, contradicting maximality.

Consider the family of edge sets 
\[\K:= \{S^-(v_1),...,S^-(v_a),T^+(u_1),\dots, T^+(u_b),e(S,w),e(w,T)\}. \] 
(here $a$ and $b$ are sizes of the source tower and target tower, respectively.)

 $\K $ has the following two properties:

\begin{itemize}
    \item Any full choice function for $\K$ contains an $s-t$ path, and
    \item For any $\K' \subseteq \K$, $|\bigcup \K'| \geq n(|\K'|-1)+1 $. 
\end{itemize}
The first condition arises since a choice function for $\{S^-(v_i)\}$ is a directed  tree spanning $V(S)$ and rooted at $s$, a choice function for $\{T^+(u_i)\}$ is a tree spanning $V(T)$ directed towards $t$, 
and a choice function for $\{e(S,w),e(w,T\}$ is a path of length 2 from $S$ to $T$. 
The second follows since all sets in $\K$ have size at least $n$, except for $e(S,w)$ and $e(w,T)$ which have at least $n+1$ combined (and crucially, both are at least 1).

Call a family $\K$ with these two properties a \emph{path enforcer}. Note that any path enforcer contains a rainbow $s-t$ path as needed: pigeonhole implies that every $\K'\subset \K$ contains edges of at least $|\K'|$ different colours in $\F'$ (as it is an $n$-scrambling of $\F$), and Hall's marriage theorem then implies there is a choice function for $\K$ which is $\F'$-rainbow, which in turn contains an $s-t$ path.

So instead, we may assume $|e(S,w)|+|e(w,T)| \leq n$ for every $w$. Since there are $>nk/2$ paths in $D$ 
leaving $S$ (and entering $T$), we can double-count as follows:
\begin{align*}
nk
&< \sum_{w \not\in S} |e(S,w)|+\sum_{w \not\in T} |e(w,T)| \\
&=\sum_{w\in W} (|e(S,w)|+|e(w,T)|)+2|e(S,T)| \\
&\leq n|W| +2|e(S,T)| \\
&\leq nk+2|e(S,T)|,
\end{align*}

meaning there is at least one edge $e$ from $S$ to $T$.

But now, $\{S^-(v_1),\dots,S^-(v_a), T^+(u_1),\dots, T^+(u_b), \{e\} \}$ is a path enforcer, so again it contains an $s-t$ path which is $\F'$-rainbow.
\end{proof}

\section{Adding a size condition}


In \cite{ab1} (Theorem 4.1)
the following strengthening of Drisko's theorem was proved:

\begin{theorem}\label{thm:driskorepeats}
Let $k \le n$, and let $M_1, \ldots ,M_{2k-1}$ be a set of $2k-1$ matchings, each of size $n$. Then there exists a matching of size $n$ contained in $\bigcup M_i$, representing at least $k$ matchings $M_i$.
\end{theorem}

Note that here, ``representing'' is injective, namely each representation is by a distinct edge.

It may also be possible to strengthen Theorem \ref{thm:driskorepeats} along the lines of Theorem \ref{stairs}.

\begin{conjecture}
Let $k \le n$, and let $M_1, \ldots ,M_{2k-1}$ be a set of $2k-1$ matchings, where $|M_i|=\min(i,k-1)$ for $i \le k-1$ and $|M_i|=n$ for $k\le i \le 2k-1$. Then there exists a matching of size $n$ contained in $\bigcup M_i$, representing at least $k$ matchings $M_i$.
\end{conjecture}

 
\begin{proof}[Proof of the case $k=2$] We have three matchings, $M_1, M_2, M_3$, where $M_1$ is a singleton edge $\{e\}$, and $|M_2|=|M_3|=n$. If $M_2 \cup M_3$ contains two disjoint cycles, then taking the $M_2$ edges from one of them and the $M_3$ edges from the other yields the desired matching. If $M_2 \cup M_3$ consists of a single cycle $C$, then $e$ partitions $C$ into two parts, and taking the $M_2$ 
edges on one side, the $M_3$ edges on the other, and adding $e$, results in the desired matching. 
\end{proof}

\section{Weighted versions}\label{sec:weighted}

Theorem \ref{weightednetworks} below is a weighted generalization of Theorem \ref{networks}. Given an edge-weighting $w$ of $D$ and a subset $E$ of its edges, let the \emph{weight} of $E$, written $w(E)$, be $\sum_{e\in E}w(e)$. 

\begin{theorem}\label{weightednetworks}
 If $|V(D)\setminus \{s,t\}|=n$ then any family $\cf$ of $n+1$   $s-t$-paths of weight $\leq k$ has a rainbow $s-t$ path of weight $\leq k$.
\end{theorem}

As a special case, note that a collection of $n+1$ $s-t$  paths of length at most $k$ have a rainbow $s-t$ path of length at most $k$ (as seen by giving a weight of 1 to every edge).

Throughout the proof, whenever a vertex $v$ is in a tree $T$ rooted at $s$, we will write $Tv$ for the directed path in $T$ from $s$ to $v$.

\begin{proof}
Construct a nested sequence of $s$-rooted trees
\[
\{s\}=T_0 \subset T_1 \subset T_2 \subset \dots \subset T_a=T
\]
as follows. For each $i$, if $t\in T_i$, the sequence terminates. Otherwise, let $\cf_i$ be the collection of all paths in $\cf$ not used in $T_i$
(as $t$ has not yet been reached, at most $n$ of the $n+1$ paths in $\cf$ are currently represented, so $\cf_i \not= \emptyset$). 
Then let $T_{i+1}$ be 
the tree
obtained from $T_i$ by adding the edge $e_i \in \bigcup \cf_i$ starting at some vertex $v\in V(T_i)$ which minimizes the quantity $w(T_iv)+w(e_i)$.

The sequence terminates at some tree $T$ containing $t$. Since $e_i\in\cf_i$ for each $i$, every $T_i$ is certainly $\cf$-rainbow, so in particular $T$ contains a rainbow $s-t$ path $Tt$.

Claim also, whenever $P \in \cf_i$ (i.e.\ $P$ is a path currently unrepresented in $T_i$) and 
$u \in V(T_i) \cap V(P)$, that
\[ w(T_i u) \leq w(P u).
\]
The theorem follows, as if the final edge $e=e_a$ reaching $t$ in $T$ had out-vertex $u$ and represented the path $P \in \cf$, then
\[
w(Tt)=w(T_{a-1}u)+w(e)\leq w(Pu)+w(e)=w(Pt)\leq k.
\]

 Prove the claim by induction on $i$. If $u=s$, then both paths are empty and their weights are 0.

If $u \neq s$, then there is an edge $e_j=vu\in T_i$ for some $j <i$ which represented a different path $Q \in \cf$.

Moreover, $P$ contains a path from $s \in T_j$ to $u \not \in T_j$, so let $x$ be the last vertex before $u$ on $P$ which is also in $T_j$, and 
$y$ be the vertex following $x$ on $P$ (so possibly $y=u$).
Then $w(T_i v)+w(e_i) \leq
w(T_j x) +w(xy)$ by choice of $e_i$, and $w(T_j x)  \leq w(Px)$ by induction (as $P$ was certainly not yet represented in $T_j$).

Hence
\[
w(T_i u)=w(T_iv)+w(e_i)
\leq w(T_j x)+w(xy)
\leq w(Px)+w(xy)
=w(Py)
\leq w(Pu),
\]
where the last inequality follows from nonnegativity of $w$, as desired.

\end{proof}

A corresponding weighted version of Theorem \ref{drisko} may also be true:

\begin{conjecture}\label{weighteddrisko}
Let $G$ be a bipartite graph with an edge-weighting. Then $2n-1$ matchings of size $n$ and weight $\leq k$ in $G$ have a rainbow matching of size $n$ and weight $\leq k$. 
\end{conjecture}

Note that in this instance, the problem is invariant under affine transformations, so dropping the non-negativity assumption on the weighting does not make the problem any harder (The weighted version of Theorem \ref{weightednetworks}, by contrast, has small counterexamples). 

Given that the original proof \cite{drisko} of Theorem \ref{drisko} was via B\'ar\'any's Theorem (Theorem \ref{barany}), one may also ask whether a weighted version also holds-in $\R^d$-do $d+1$ $\vv$-convexing sets in of weight $\leq k$ have a rainbow $\vv$-convexing set of weight $\le k$?


    
\section{Cooperative  versions }\label{coopsection}
Hall's theorem is ``cooperative'' in the sense that its assumption is not on individual sets, but on unions of sets. 
There are also cooperative versions of the second type of choice functions results - spanning rainbow sets. Here is  a natural conjectured cooperative version of Theorem \ref{drisko}, which also generalizes Theorem \ref{stairs}:
    
    \begin{conjecture}\label{halldrisko}
\label{coopdriko}
Let $\F=(F_1, \ldots , F_{2k-1})$ be  a system of  sets of edges in a bipartite graph.  If  $\nu(\F_I) \ge \min(|I|,k)$ for every $I \subseteq [2k-1]$  then there exists a rainbow matching of size $k$.
\end{conjecture}

Theorem \ref{drisko} is the special case where  $\nu(F_i) \geq k$ for every $i$. The cooperative version in which  $\nu(F_{\{i,j\}}) \geq k$ whenever $i \neq j$,
is Theorem \ref{coopdrisko}  below.

 A fractional version of the conjecture was proved in \cite{abkz}. 
A topological version is: 

\begin{conjecture}
Let $G$ be a bipartite graph, and let $M(G)$ be the matching complex of $G$.
Suppose that one side of the graph is of size $2n-1$, and that every $k$ vertices in that side have at least $\min(k,n)$ neighbors in the other side. Then $\eta(M(G)) \ge n$.   
\end{conjecture}

   Here is a cooperative conjecture on networks.  Given a set $F \subseteq E(\N)$ let $R(F)$ be the set of vertices $x$ having an $F$-path from $s$ to $x$.

    \begin{conjecture}
    Let $n=|V^\circ(D)|$, and let $\cf=(F_1, \ldots ,F_{n+1})$ be a family of sets of edges in the network.
    If, for every $I \subseteq [n+1]$, either  $\cf_I$ contains an $s-t$-path or $|R(\cf_I)|\ge |I|$ 
    then there exists a rainbow $s-t$-path. 
            \end{conjecture}

            The complex in action here is the collection $NREACH$ of sets of edges, that do not contain an $s-t$ path.
            It is not hard to show that $\blambda(NREACH) \le n$, which can be combined with Theorem \ref{km} to yield Theorem \ref{networks}.  
            
            \begin{conjecture}
            $\blambda(NREACH)$ is the maximal number of vertices in a family of innerly disjoint $s-t$ paths. 
            \end{conjecture}

Is there a cooperative version of Theorem \ref{barany}?
Somewhere over the rainbow there might well be one. Presently, only  a few humble results exist in this direction. One of them concerns pairwise cooperation:

\begin{theorem}\label{hpt}\cite{hpt}
\begin{enumerate}
    \item Let $S_1, \ldots, S_{d+1}$ be subsets of $\mathbb{R}^d$, and let $\vec{v} \in \mathbb{R}^d$. If for every $1\le i<j\le d+1$ we have $\vec{v} \in \conv(A_i \cup A_j)$ then there exists a rainbow set $R$  such that $\vec{v}  \in \conv(R)$. 
    
    \item Let $S_1, \ldots, S_{d}$ be non-empty sets in $\mathbb{R}^d$, 
all contained in the same closed half-space. Suppose that  for every $1\le i<j\le d$ we have $v\in \cone(A_i \cup A_j)$. Then there exists a rainbow set $R$ for the sets $A_i$, such that $v \in \cone(R)$.

\end{enumerate}

\end{theorem}

A similar extension exists for Theorem \ref{drisko}. For a system $\F$ of sets of edges, denote by $\nu_r(\F)$ the maximal size of a rainbow matching. 

\begin{theorem}\label{coopdrisko}\cite{abck}
Let $\F=(F_1, \ldots , F_{2k-1})$ be  a system of bipartite sets of edges, sharing the same bipartition. If each $F_i$ is non-empty, and  $\nu(F_i\cup F_j)\ge k$ for every $i<j \le 2k-1$, then $\nu_r(\F)\ge k$.
\end{theorem}

The proof of Theorem \ref{coopdrisko} uses a cooperative result on networks, which we omit. 

J. Kim \cite{jinhaprivate} proved a cooperative version of Theorem \ref{kmleray}:

\begin{theorem}\label{jinha}
Let $\C$ be a simplicial complex and let $V=V(\C)$. Assume that $V \not \in \C$. Let $\cs=(S_1, \ldots,S_m)$ be a system of subsets of $V$. Suppose that for each
$I \subseteq [m]$ either  $\cs_I \not \in \C$ or $lk_\C(\cs_I) $ is $m-1-|I|$-Leray. Then $\cs$ has a rainbow set not belonging to $\C$.
\end{theorem}
 
(Here $lk_\C(S)$
is the link of $S$ in $\C$.)
Using this theorem it is possible to prove another cooperative generalization of Theorem \ref{barany}

\begin{theorem}
Let $t \ge 1$, If $S_1, \ldots ,S_{d+t}$ are subsets of $\mathbb{R}^d$ satisfy the condition that for every $I\subseteq [d+t]$ there holds $\vec{v} \in conv(S_I)$ then there exists 
a rainbow set $S$ such that $\vec{v} \in conv(S)$.
\end{theorem}

 Theorem \ref{barany} is the case $t=1$.

  \subsection{Cooperative conditions for spanning rainbow sets}

Here is a cooperative version of Theorem 
 \ref{spanning}.

\begin{theorem}
    Let $\M$ be a matroid of rank $n$,  let $T \subseteq V(\cm)$, and let $\A=(A_1, \ldots,A_n)$ be a system of subsets of $V(\M)$, satisfying the following: for every $J \subseteq [n]$, either $\rank(A_J)\ge |J|$ or $v \in span_\M(A_J)$. Then there exists an $\A$-rainbow set spanning  $T$. 
\end{theorem}
This is a common 
generalization of Rado's theorem (Theorem 
\ref{rado}), which is the case $T=V$,  and Theorem 
\ref{spanning}
which is the case in which all $J$s satisfy the second condition.

\begin{proof}
    As noted, we may assume that there exists $J \subset [n]$ for which $\rank(\A_J) < |J|$, or else by Rado's theorem (Theorem \ref{rado}) there exists a rainbow base. Choose a minimal such $J$.

    Certainly $|J| \geq \rank(\A_J)+1 \geq 1$, so choose any $j \in J$. Applying Rado's theorem to $I:=J \backslash \{j\}$ yields then a  rainbow $R \in \M$ representing  sets $A_i$ for all $i \in I$.
    By the assumption of the theorem 
    $T \subseteq \vspan(\A_J)$, so it suffices to prove that $ \vspan(\A_J) = \vspan(R)$.
       Clearly $\vspan(R) \subseteq \vspan(\A_J)$, and since $\rank(R)=\rank(\A_J)=|J|-1$, this inclusion is in fact equality. 
\end{proof}

This yields a cooperative version of Theorem \ref{main}. Denote by $\C(G)$ is the set of components of a graph $G$. 

\begin{theorem}\label{refinedmain}
Let $F_1, \ldots ,F_n$ be sets of edges in a graph with $n$ vertices, satisfying the following condition: for every $J \subseteq [n]$ either
\begin{itemize}
    \item 
$\sum_{C \in \C( F_J)}(|V(C)|-1) \ge |J|$, or
\item $F_J$ contains an odd cycle. 
\end{itemize}
Then there exists a rainbow odd cycle. 
\end{theorem}

\section{Rainbow independent sets in hypergraphs} 

We conclude with a rather neglected topic: rainbow independent sets  in hypergraphs (to be distinguished from rainbow matchings in hypergraphs, which are independent sets in line graphs). 

A set $I$ of vertices is said to be {\em independent} in a hypergraph $H$ if it does not contain any edge of $H$. By $\alpha(H)$ we denote the largest size of an independent set. 
In \cite{carotuza} and in \cite{shachnai} the following was proved:

\begin{theorem}
 $\alpha(H) \ge \frac{|V(H)|}{d^{1/{k-1}}}$
\end{theorem}

 Using the Lov\'asz Local Lemma it is possible to prove:

\begin{theorem}
 If $H$ is $k$-uniform with maximal vertex degree  $d$, then any collection  of disjoint sets of vertices of $H$, each  of size at least $ked^{1/{k-1}}$,  has a full rainbow independent set.
\end{theorem}

But the coefficient $ked$ can probably be replaced by a smaller one, for example by a result of Haxell \cite{penny}, for $k=2$ (namely for graphs) the coefficient is  $2d$, rather than $2ed$. Here is a conjecture for $k=3$, stemming from examples whose description we omit:

\begin{conjecture}
If $H$ is $3$-uniform with maximal  degree  $d$, then any collection  of disjoint sets of vertices of $H$, each  of size at least $1.1\sqrt{d}$,  has a full rainbow independent set.   
\end{conjecture}

Using the Lov\'asz Local Lemma it is possible to prove this result with $1.1 \sqrt{d}$ replaced by $3e \sqrt{d}$.

{\bf Data Availability} No data was needed to produce the results in this paper.

{\bf Acknowledgements} We are indebted for fruitful discussions with Eli Berger, Ron Holzman and Zilin Jiang. We are also indebted to Dmitry Falikman for computer tests of some of the  conjectures, and useful  insights. 


\end{document}